\documentclass[11pt,reqno]{amsart}
\usepackage{amsmath,amsfonts,amscd,amssymb,latexsym,mathrsfs,mathtools}
\usepackage[left=2.5cm,right=2.5cm,top=2.5cm,bottom=2.5cm]{geometry}

\usepackage[nobysame]{amsrefs}
\usepackage[percent]{overpic} 
\usepackage{enumitem}
\usepackage{comment}
\usepackage{soul}
\usepackage{tikz}
\usepackage{marginnote}
\usepackage{bbm}
\usepackage{cancel,slashed}
\usepackage[normalem]{ulem}

\usepackage{footmisc}
\usepackage{hyperref}
\hypersetup{
    colorlinks=true,
    linkcolor=blue,
    citecolor=green,}

\newtheorem{theorem}{Theorem}[section]
\newtheorem{lemma}[theorem]{Lemma}
\newtheorem{proposition}[theorem]{Proposition}
\newtheorem{corollary}[theorem]{Corollary}
\newtheorem{question}[theorem]{Question}%

\theoremstyle{definition}
\newtheorem{definition}[theorem]{Definition}

\newtheorem{assumption}[theorem]{Assumption}

\theoremstyle{remark}
\newtheorem{remark}[theorem]{Remark}

\newcommand{\CC}{\mathbb{C}}

\newcommand{\RR}{\mathbb{R}}

\newcommand{\bfs}{\mathbf{s}}

\newcommand{\calA}{\mathcal{A}}

\newcommand{\calD}{\mathcal{D}}

\newcommand{\calP}{\mathcal{P}}
\newcommand{\calR}{\mathcal{R}}

\newcommand{\frakt}{\mathfrak{t}}

\newcommand{\id}{\operatorname{id}}
\newcommand{\dom}{\operatorname{dom}}

\newcommand{\double}{\mathsf{D}}
\newcommand{\D}{\mathrm{d}}

\newcommand{\scal}{{\rm scal}}
\newcommand{\rmc}{\mathrm{c}}
\newcommand{\rmS}{\mathrm{S}}
\newcommand{\rmT}{\mathrm{T}}

\newcommand{\upper}{\uppercase\expandafter}
\newcommand{\romnu}{\romannumeral}

\newcommand{\p}{\partial}
\newcommand{\pM}{{\p M}}

\newcommand{\supp}{\operatorname{supp}}

\newcommand{\End}{\operatorname{End}}
\newcommand{\rank}{\operatorname{rank}}
\newcommand{\dist}{\operatorname{dist}}

\newcommand{\cut}{{\rm cut}}

\newcommand{\Tr}{\operatorname{Tr}}

\newcommand{\spec}{\operatorname{spec}}
\newcommand{\spf}{\operatorname{sf}}

\newcommand{\IM}{\operatorname{im}}

\newcommand{\slaD}{\slashed{D}}
\newcommand{\slaS}{\slashed{S}}
\newcommand{\tilD}{\tilde D}
\newcommand{\tilE}{\tilde E}
\newcommand{\tilF}{\tilde F}

\newcommand{\tilS}{\tilde S}

\newcommand{\loc}{{\rm loc}}

\begin{document}

\normalsize

\title[Odd-dimensional long neck problem via spectral flow]{The odd-dimensional long neck problem via spectral flow}

\author[Pengshuai Shi]{Pengshuai Shi${}^\ast$}
\address{School of Mathematics and Statistics, Beijing Institute of Technology, Beijing 100081, P. R. China}

\email{shipengshuai@bit.edu.cn, pengshuai.shi@gmail.com}




\begin{abstract}
In this paper, we establish a scalar-mean curvature comparison theorem for the long neck problem on odd-dimensional spin manifolds. This extends previous work of Cecchini and Zeidler, and gives a complete answer to Gromov's long neck problem in terms of spin manifolds. As a related question, we prove a quantitative version of Llarull's theorem on non-compact spin manifolds. Our results are derived by studying the spectral flow of a family of Callias operators.
\end{abstract}

\maketitle


\section{Introduction}\label{S:intro}

A central theme in differential geometry is the study of Riemannian manifolds with various curvature conditions. One of the actively-developing directions is the study of manifolds with positive scalar curvature. On spin manifolds, there is a fundamental method based on Lichnerowicz vanishing theorem and index theory of spin Dirac operators to deal with such questions. When the manifold is equipped with an area-decreasing map structure, we have the following Llarull's rigidity theorem.

\begin{theorem}[\cite{Llarull98}]\label{IT:Llarull}
Let $(M,g)$ be an $n$-dimensional ($n\ge 3$) (for 2-dimensional case, the map $f$ should be distance-decreasing) connected closed Riemannian spin manifold such that the scalar curvature $\scal_g\ge n(n-1)$. Suppose there exists a smooth area-decreasing map $f$ from $M$ to the standard unit sphere $\rmS^n$ (which means $|f^*\omega|\le|\omega|$ for any two-form $\omega$ on $\rmS^n$) of non-zero degree. Then $f$ is an isometry.
\end{theorem}

This is an important result in the comparison geometry of scalar curvature. The proof utilizes the method on spin manifolds mentioned above. There are several generalizations of Llarull's theorem, including \cite{GoetteSem02} with the sphere replaced by a manifold with non-negative curvature operator, \cite{Lott21pams} for manifolds with boundary, \cite{Su19,SuWangZhang22} for foliated manifolds, \cite{CHS22lip} for metrics of low regularity and Lipschitz continuous maps, \cite{CWXZ24} for 4-dimensional non-spin manifolds (in the case of distance-decreasing maps), and \cite{Zhang20,LiSuWangZhang24} for non-compact manifolds (see Theorem~\ref{IT:Llarull non-cpt} below), etc.

In an earlier version of \cite{Gromov23Four}, Gromov proposed a question about metric inequalities in this area-decreasing map setting on manifolds with boundary, which is called the \emph{long neck problem}.

\begin{question}[Long neck problem]\label{IQ:long neck}
Let $(M,g)$ be an $n$-dimensional connected compact Riemannian manifold with boundary. Let $f:M\to\rmS^n$ be a smooth area-decreasing map that is locally constant near the boundary. Suppose $\scal_g\ge n(n-1)$ on $M$. Find a lower bound for the distance between $\pM$ and $\supp(\D f)$ which forces the degree of $f$ to be zero.
\end{question}

In the realm of spin manifolds, by developing an index theory on incomplete manifolds, Cecchini \cite{Cecchini20LN} obtained the following long neck principle.

\begin{theorem}[{\cite[Theorem~A]{Cecchini20LN}}]\label{IT:long neck Cecchini}
Let $(M,g)$ be an $n$-dimensional connected compact Riemannian spin manifold with boundary. Let $f:M\to\rmS^n$ be a smooth strictly area-decreasing map that is locally constant near the boundary. When $n$ is odd, $f$ is further required to be constant near the boundary. Suppose $\scal_g\ge n(n-1)$ on $\supp(\D f)$ and $\scal_g\ge\sigma^2n(n-1)$ on $M$ for some $\sigma>0$. If
\[
\dist_g(\supp(\D f),\pM)>\frac{\pi}{\sigma n},
\]
then $\deg(f)=0$.
\end{theorem}

Later in \cite{CeccZeid24GT}, Cecchini and Zeidler exploited a refined approach to this question (and others) using the method of index theory for Callias operators under local boundary conditions. The advantage of their approach is that they are able to bring the mean curvature of the boundary into the picture and provide a deeper understanding of the long neck problem. To be precise, they proved the following scalar-mean curvature comparison theorem for the long neck problem.

\begin{theorem}[{\cite[Theorem~1.4]{CeccZeid24GT}}]\label{IT:long neck CeccZeid}
Let $(M,g)$ be an $n$-dimensional ($n\ge2$ even) connected compact Riemannian spin manifold with boundary. Let $f:M\to\rmS^n$ be a smooth area-decreasing map that is locally constant near $\pM$. Suppose $\scal_g\ge n(n-1)$ and the mean curvature (in this paper, we adopt the convention that the unit sphere $\rmS^{n-1}$ has mean curvature 1 as the boundary of the unit ball in $\RR^n$) $H_g\ge-\tan(\frac{1}{2}nl)$ for some $l\in(0,\frac{\pi}{n})$. If
\[
\dist_g(\supp(\D f),\pM)\ge l,
\]
then $\deg(f)=0$.
\end{theorem}

Theorem~\ref{IT:long neck CeccZeid} is stronger than Theorem~\ref{IT:long neck Cecchini} and together they provide a satisfactory answer to the long neck problem for even-dimensional spin manifolds. These results are generalized recently by Liu in \cite{LiuDQ24pams,LiuDQ24mathz}. However, some issues are left for the odd-dimensional case. On one hand, there is a more restrictive assumption of $f$ being \emph{constant} instead of locally constant near the boundary in Theorem~\ref{IT:long neck Cecchini}. As Gromov commented in \cite[p.257]{Gromov23Four}, ``this is, probably, redundant''. On the other hand, Theorem~\ref{IT:long neck CeccZeid} is yet to be established in odd dimensions.

In this paper, we affirm the validity of Theorem~\ref{IT:long neck CeccZeid} for odd-dimensional spin manifolds in the theorem below.

\begin{theorem}\label{IT:long neck}
Let $(M,g)$ be an $n$-dimensional ($n\ge3$ odd) connected compact Riemannian spin manifold with boundary. Let $f:M\to\rmS^n$ be a smooth area-decreasing map that is locally constant near the boundary. Suppose $\scal_g\ge n(n-1)$ on $\supp(\D f)$, $\scal_g\ge\sigma^2n(n-1)$ on $M$ for some $\sigma>0$, and $H_g\ge-\sigma\tan(\frac{1}{2}\sigma nl)$ for some $l\in(0,\frac{\pi}{\sigma n})$. If 
\[
\dist_g(\supp(\D f),\pM)\ge l,
\]
then $\deg(f)=0$.
\end{theorem}

Note that $\tan(\frac{1}{2}\sigma nl)\to\infty$ as $l\to\frac{\pi}{\sigma n}$. As a consequence, for odd-dimensional spin manifolds, the aforementioned extra assumption can be dropped and there holds the improved long neck principle.

\begin{corollary}\label{IC:long neck}
Let $(M,g)$ be an $n$-dimensional ($n\ge3$ odd) connected compact Riemannian spin manifold with boundary. Let $f:M\to\rmS^n$ be a smooth area-decreasing map that is locally constant near the boundary. Suppose $\scal_g\ge n(n-1)$ on $\supp(\D f)$ and $\scal_g\ge\sigma^2n(n-1)$ on $M$ for some $\sigma>0$. If
\[
\dist_g(\supp(\D f),\pM)\ge\frac{\pi}{\sigma n},
\]
then $\deg(f)=0$.
\end{corollary}

\begin{remark}\label{IR:long neck sharp}
Cecchini and Zeidler have shown in \cite[Proposition~5.2]{CeccZeid24GT} that Theorems~\ref{IT:long neck CeccZeid} and \ref{IT:long neck} (for $\sigma=1$) are sharp (regardless of dimension parity), which means that the constant $\frac{\pi}{n}$ in Theorem~\ref{IT:long neck Cecchini} and Corollary~\ref{IC:long neck} (for $\sigma=1$) is optimal. Their example is constructed from the toric band $\rmT^{n-1}\times[-l,l]$ considered by Gromov in \cite{Gromov18metric} with warped product metric $g=\varphi^2g_{\rmT^{n-1}}+\D t^2$ for $\varphi(t)=\cos^{2/n}(\frac{1}{2}nt)$. This also shows the optimality of $\frac{\pi}{\sigma n}$ for $\sigma>1$. It is not clear to the author whether the case of $0<\sigma<1$ is optimal as well.
\end{remark}

Our proof of Theorem~\ref{IT:long neck} uses Callias operators with Lipschitz potential from a Gromov--Lawson pair, whose index was studied by Cecchini--Zeidler in \cite{CeccZeid24GT}. But for odd-dimensional manifolds, the index is not a suitable object to work with. This is because the argument here involves the (complex) spin Dirac operator on a closed manifold twisted by a Hermitian bundle, whose index vanishes identically in odd dimensions. Inspired by \cite{LiSuWang24}, where Li, Su and Wang presented a new proof of Llarull's rigidity theorem via spectral flow, we turn to consider the spectral flow approach. With the help of a splitting formula (Theorem~\ref{T:splitting spf}), we are able to deal with the spectral flow of a family of Callias operators on manifolds with boundary. Then we can conduct a standard contradiction argument.

Intuitively, the degenerate case $\sigma=0$ in Corollary~\ref{IC:long neck} would suggest that infinitely long neck forces the degree of $f$ to be zero. In fact, this rough observation corresponds to the following Llarull's theorem for non-compact manifolds of Zhang \cite{Zhang20}, and Li, Su, Wang and Zhang \cite{LiSuWangZhang24}.

\begin{theorem}[\cite{Zhang20,LiSuWangZhang24}]\label{IT:Llarull non-cpt}
Let $(M,g)$ be an $n$-dimensional connected non-compact complete Riemannian spin manifold without boundary. Let $f:M\to\rmS^n$ be a smooth area-decreasing map that is locally constant near infinity and of non-zero degree. Suppose that $\scal_g\ge n(n-1)$ on $\supp(\D f)$. Then $\inf(\scal_g)<0$.
\end{theorem}

The non-strict inequality $\inf(\scal_g)\le0$ can be proved in a relatively easier way by using Gromov--Lawson's relative index theorem. For the strict inequality, Zhang \cite{Zhang20} considered a deformed Dirac operator and then applied the relative index theorem to prove the theorem in even-dimensional case and a weaker form in odd-dimensional case. In \cite{LiSuWangZhang24}, the authors combined ideas in \cite{Zhang20} and \cite{LiSuWang24} and converted the problem to a closed manifold. In this way they obtained a proof for odd-dimensional manifolds. 
Using the method to prove Theorem~\ref{IT:long neck}, we can get a refined quantitative version of Theorem~\ref{IT:Llarull non-cpt}.

\begin{theorem}\label{IT:quanti Llarull}
Let $(M,g)$ be an $n$-dimensional connected non-compact complete Riemannian spin manifold without boundary. Let $f:M\to\rmS^n$ be a smooth area-decreasing map which is locally constant near infinity and of non-zero degree. Let $K\Subset M$ be a compact subset containing $\supp(\D f)$ with smooth boundary such that $0<\delta:=\dist_g(\supp(\D f),\p K)<\frac{\pi}{\sigma n}$ for some $\sigma>0$.
Suppose $\scal_g\ge n(n-1)$ on $\supp(\D f)$ and $\scal_g\ge\sigma^2n(n-1)$ on $K\setminus\supp(\D f)$. Then,
\[
\inf(\scal_g)<-\sigma^2n(n-1)\tan^2\Big(\frac{1}{2}\sigma n\delta\Big).
\]
\end{theorem}

From this theorem, one concludes that the faster the scalar curvature near $\supp(\D f)$ decays, the larger the infimum of the scalar curvature on $M$ could be. But it is always negative. In particular, we get another proof of Theorem~\ref{IT:Llarull non-cpt} in odd dimensions based on spectral flow of Callias operators (see Section~\ref{S:quanti Llarull}).

The techniques developed in the current paper have been applied to deal with other problems related to scalar curvature, including the band width estimate in a later work \cite{Shi25oddK}.

\subsection*{Organization of the paper}\label{SS:orga}

This paper is organized as follows. In Section~\ref{S:splitting spf}, we introduce a splitting formula for the spectral flow of a family of self-adjoint Fredholm Dirac-type operators on complete Riemannian manifolds. In Section~\ref{S:Callias}, we discuss the main object of the paper, Callias operators with Lipschitz potential associated to a Gromov--Lawson pair, and obtain a formula computing its spectral flow on manifolds with boundary. In Section~\ref{S:long neck}, we prove the scalar-mean curvature comparison theorem for the long neck problem in odd dimensions (Theorem~\ref{IT:long neck}). Section~\ref{S:quanti Llarull} is devoted to the proof of the quantitative Llarull's theorem on non-compact manifolds (Theorem~\ref{IT:quanti Llarull}) and a new proof of Llarull's theorem on non-compact manifolds (Theorem~\ref{IT:Llarull non-cpt}).

\section{A splitting formula for the spectral flow}\label{S:splitting spf}

In this section, we present a splitting formula for the spectral flow of a continuous path of self-adjoint Dirac-type operators, which relates the spectral flow on the original manifold to that on the partitioned manifold. Similar results in compact situation have appeared in different forms in the literature, for example \cite{Bunke95eta,Nicolaescu95,CappellLeeMiller96-2,DanielKirk99,FurutaniOtsuki02,Furutani06aps-spf}, and most notably \cite{KirkLesch04}, often assuming product structure near the boundary of the partitioned manifold. Here we formulate our formula for possibly non-compact manifolds without this assumption and in a form that is suitable for our application in later sections.

\subsection{Self-adjoint Fredholm boundary conditions for Dirac-type operators}\label{SS:s.a-Fred Dirac}

Let $M$ be a complete Riemannian $n$-manifold possibly with boundary. A Hermitian vector bundle $S\to M$ is called a \emph{Dirac bundle} if there is a Clifford multiplication $\rmc(\cdot):T^*M\to\End(S)$ that is skew-adjoint and satisfies $\rmc(\cdot)^2=-|\cdot|^2$, and a Hermitian connection $\nabla$ that is compatible with $\rmc(\cdot)$ (i.e. $\rmc(\cdot)$ is a parallel bundle endomorphism, cf. \cite[\S II.5]{LawMic89}). The \emph{Dirac operator} is a formally self-adjoint first-order differential operator acting on sections of a Dirac bundle, defined by
\[
D:=\sum_{i=1}^{n}\rmc(e_i^*)\nabla_{e_i}:\ C^\infty(M,S)\to C^\infty(M,S),
\]
where $e_1,\dots,e_n$ is an orthonormal local tangent frame and $e_1^*,\dots,e_n^*$ is the associated dual cotangent frame. We use the convention that the principal symbol of a Dirac operator is the Clifford multiplication. An operator that has the same principal symbol as a Dirac operator is called a \emph{Dirac-type operator}.

Recall the following \emph{Bochner--Schr\"odinger--Lichnerowicz--Weitzenb\"ock formula} (cf. \cite[\S II.8]{LawMic89})
\begin{equation}\label{E:BSLW for}
D^2=\nabla^*\nabla+\calR,
\end{equation}
where $\nabla^*\nabla$ is the connection Laplacian on $S$ and
\begin{equation}\label{E:BSLW endo}
\calR=\sum_{i<j}\rmc(e_i^*)\rmc(e_j^*)R^S(e_i,e_j)
\end{equation}
is a curvature endomorphism of the curvature tensor $R^S=(\nabla)^2$ of $S$.

Let $\calP$ be the \emph{Penrose operator} defined by
\[
\calP_eu:=\nabla_eu+\frac{1}{n}\rmc(e^*)Du,
\]
for any $e\in TM,u\in C^\infty(M,S)$. Then we have (cf. \cite[\S 5.2]{BHMMM15book})
\begin{equation}\label{E:Penrose}
|\nabla u|^2=|\calP u|^2+\frac{1}{n}|Du|^2.
\end{equation}

\begin{definition}\label{D:coercive}
A Dirac-type operator $\calD$ is said to be \emph{coercive at infinity} if there exist a compact subset $K\Subset M$ and a constant $C>0$ such that
\[
\|\calD u\|_{L^2(M,S)}\ge C\|u\|_{L^2(M,S)}
\]
for any smooth section $u$ with compact support in $M\setminus K$.
\end{definition}

Consider a Dirac-type operator $\calD$ on a complete Riemannian manifold $M$ with compact boundary. 
Then there exists a formally self-adjoint differential operator $\calA:C^\infty(\p M,S_{|\p M})\to C^\infty(\p M,S_{|\p M})$ of first-order with principal symbol
\[
\sigma_\calA(\xi)=\rmc(\nu^*)^{-1}\rmc(\xi),
\]
where $\xi\in T^*\p M$, and $\nu^*$ is the dual covector of the inward pointing unit normal vector field $\nu=e_n$ along $\p M$. It is also a Dirac-type operator. In addition, one can further require that $\calA$ anti-commutes with $\rmc(\nu^*)$ (cf. \cite[Section~3]{BaerBallmann16}). Such an operator $\calA$ is called a \emph{compatible adapted operator} to $\calD$.

\begin{remark}\label{R:adatped op}
A compatible adapted operator can be given by the canonical boundary Dirac operator. That is, we make $S_{|\pM}$ a Dirac bundle by setting
\[
\begin{aligned}
\rmc^\p(\xi)&=\rmc(\nu^*)^{-1}\rmc(\xi),\quad\text{for }\xi\in T^*\p M, \\
\nabla^\p&=\nabla+\frac{1}{2}\rmc^\p(\nabla\nu^*).
\end{aligned}
\]
It can be checked that the Dirac operator
\[
A:=\sum_{i=1}^{n-1}\rmc^\p(e_i^*)\nabla^\p_{e_i}:\ C^\infty(\pM,S_{|\pM})\to C^\infty(\pM,S_{|\pM})
\]
anti-commutes with $\rmc(\nu^*)$. Also, we have
\[
A=\rmc(\nu^*)^{-1}D-\nabla_\nu+\frac{n-1}{2}H,
\]
where $H$ is the mean curvature of $\pM$ with respect to $\nu$. See \cite[Appendix~1]{BaerBallmann16}.
\end{remark}

Let $P:L^2(\pM,S_{|\p M})\to L^2(\pM,S_{|\p M})$ be an orthogonal projection. If $P$ is a pseudo-differential operator of order 0, then $P$ defines a boundary condition for $\calD$. We denote by $\calD_P$ the operator $\calD$ with this boundary condition, whose domain is
\[
\dom\calD_P:=\{u\in H^1_\calD(M,S)\;|\;P(u_{|\p M})=0\},
\]
where
\[
H^1_\calD(M,S):=\{u\in H^1_\loc(M,S)\cap L^2(M,S)\;|\;\calD u\in L^2(M,S)\},
\]
and $u_{|\pM}$ denotes the extension of the boundary restriction map defined on smooth sections to the Sobolev space $H^1_\calD$.
We call the boundary condition defined by $P$ a \emph{pseudo-local} boundary condition. If $S_1\subset S_{|\p M}$ is a subbundle, and $P$ is the projection induced by the fiberwise orthogonal projection onto $S_1$, then $P$ defines a \emph{local} boundary condition for $\calD$.

Fix a compatible adapted operator $\calA$. Let $\Pi_{\ge0}:L^2(\pM,S_{|\p M})\to L^2(\pM,S_{|\p M})$ be the non-negative spectral projection for the operator $\calA$. (Notice that $\calA$ is a formally self-adjoint Dirac-type operator on a closed manifold $\p M$, thus is self-adjoint. So $\calA$ has discrete spectrum in $\RR$, and the unit eigensections of $\calA$ form an orthonormal basis of $L^2(\p M,S_{|\pM})$.)

\begin{definition}[{\cite[Definition~2.1]{KirkLesch04}}]\label{D:s.a-Fred}
The \emph{self-adjoint Fredholm Grassmannian} ${\rm Gr}(\calA)$ is defined to be the set of orthogonal projections $P:L^2(\pM,S_{|\p M})\to L^2(\pM,S_{|\p M})$ such that
\begin{enumerate}
\item $P$ is a pseudo-differential operator of order 0; \label{item:s.a-Fred 1}
\item $\rmc(\nu^*)P\rmc(\nu^*)^{-1}=\id-P$; and \label{item:s.a-Fred 2}
\item $\Pi_{\ge0}|_{\IM P}:\IM P\to\IM\Pi_{\ge0}$ is a Fredholm operator. \label{item:s.a-Fred 3}
\end{enumerate}
\end{definition}

The following proposition is a consequence of the theory of boundary value problems for Dirac-type operators due to B\"ar--Ballmann \cite{BaerBallmann12,BaerBallmann16}.

\begin{proposition}\label{P:s.a-Fredholm}
Let $\calD$ be a formally self-adjoint Dirac-type operator that is coercive at infinity. Then for any $P\in{\rm Gr(\calA)}$, $\calD_P$ is a self-adjoint Fredholm operator.
\end{proposition}

In fact, by \cite[Theorem~7.20]{BaerBallmann12}, a pseudo-local boundary condition given by $P$ is elliptic if and only if $P$ satisfies Definition~\ref{D:s.a-Fred}.\ref{item:s.a-Fred 3}. Condition~\ref{item:s.a-Fred 2} of Definition~\ref{D:s.a-Fred} implies that this boundary condition is self-adjoint. And the Fredholmness follows from \cite[Corollary~8.6]{BaerBallmann12}.

\subsection{Boundary conditions on partitioned manifolds}\label{SS:partition}

Let $\calD$ be a Dirac-type operator acting on a Dirac bundle $S$ over a complete Riemannian manifold $M$. Without loss of generality, we assume that $M$ is without boundary. Let $\Sigma$ be a closed hypersurface of $M$ with trivial normal bundle. Cutting $M$ along $\Sigma$, we get a manifold $M^\cut$ whose boundary consists of two copies $\Sigma_1$ and $\Sigma_2$ of $\Sigma$. The Dirac-type operator $\calD$ naturally induces a Dirac-type operator $\calD^\cut:C^\infty(M^\cut,S_{|M^\cut})\to C^\infty(M^\cut,S_{|M^\cut})$.

In this setting, we can define the \emph{continuous transmission condition} for $\calD^\cut$, which is given by the domain
\[
\dom(\calD^\cut_{P_\Delta}):=\left\{u\in H^1_\calD(M^\cut,S_{|M^\cut})\;|\;\,u_{|\p M^\cut}=(f,f)\in L^2(\Sigma,S_{|\Sigma})\oplus L^2(\Sigma,S_{|\Sigma})\right\},
\]
under the canonical identification
\[
L^2(\p M^\cut,S_{|\p M^\cut})=L^2(\Sigma,S_{|\Sigma})\oplus L^2(\Sigma,S_{|\Sigma}).
\]
It is shown in \cite[Example~7.28]{BaerBallmann12} that this condition is an elliptic boundary condition. Equivalently, the continuous transmission condition is given by the continuous transmission projection
\[
P_\Delta=\frac{1}{2}\left(
\begin{matrix}
1 & -1 \\
-1 & 1
\end{matrix}\right).
\]
Clearly, $\calD^\cut_{P_\Delta}$ can be canonically identified with $\calD$.

\begin{remark}\label{R:trans b.c}
Let $\calA_\Sigma$ be a compatible adapted operator for $\calD^\cut$ on $\Sigma_1$. Then $-\calA_\Sigma$ is a compatible adapted operator on $\Sigma_2$. So $\calA=\calA_\Sigma\oplus-\calA_\Sigma$ is a compatible adapted operator for $\calD^\cut$ on $\p M^\cut$. As pointed out in \cite[p. 572]{KirkLesch04}, since $P_\Delta$ has off-diagonal terms, it is not a pseudo-differential operator on $S_{|\p M^\cut}$. But it is pseudo-differential on the bundle $S_{|\Sigma}\oplus S_{|\Sigma}$ over $\Sigma$.
In the discussion below, we will ignore this distinction, as it does not affect the argument.
\end{remark}

Note that if $Q\in{\rm Gr}(\calA_\Sigma)$, then $\id-Q\in{\rm Gr}(-\calA_\Sigma)$. So
\begin{equation}\label{E:s.a-b.c}
P:=\left(
\begin{matrix}
Q & 0 \\
0 & \id-Q
\end{matrix}\right)\in{\rm Gr}(\calA).
\end{equation}
For any $Q\in{\rm Gr}(\calA_\Sigma)$, we can connect $P_\Delta$ and $P$ by a path (cf. \cite[Section~3]{BruningLesch99})
\begin{equation}\label{E:s.a-b.c path}
P(s):=\left(
\begin{matrix}
Q\cos^2s+(\id-Q)\sin^2s & -\cos s\sin s \\
-\cos s\sin s & (\id-Q)\cos^2s+Q\sin^2s
\end{matrix}\right),\quad 0\le s\le\frac{\pi}{4},
\end{equation}
so that $P(0)=P$ and $P(\pi/4)=P_\Delta$.

\begin{lemma}[{\cite[Lemma~5.3]{KirkLesch04}}]\label{L:path s.a-b.c}
For any $s\in[0,\pi/4]$, $P(s)\in{\rm Gr}(\calA)$. Therefore, if $\calD$ is formally self-adjoint and coercive at infinity, then $\calD_s:=\calD^\cut_{P(s)}$ is a continuous family of self-adjoint Fredholm operators in the graph topology (see \cite{GorokhovskyLesch15}, also \cite{BoossLeschPhillips05,Lesch05sf}).
\end{lemma}

\begin{remark}\label{R:path s.a-b.c}
Here we view $\calD_s$ as a family of operators on a fixed domain $\dom\calD_0$. To be precise, there is a continuous family of unitary operators $\Omega_s$ on $L^2(M^\cut,S_{|M^\cut})$ as in \cite[Section~3]{BruningLesch99} (see also \cite[Section~6]{Shi22}) that map $\dom\calD_0$ to $\dom\calD_s$. By conjugating $\calD_s$ by $\Omega_s$, we get a norm continuous family of self-adjoint Fredholm operators from $\dom\calD_0$ to $L^2(M^\cut,S_{|M^\cut})$.
\end{remark}

\subsection{Spectral flow}\label{SS:spf}

Consider a self-adjoint Dirac-type operator $\calD$ which is coercive at infinity, acting on a Dirac bundle $S$ over a complete Riemannian manifold $M$ (possibly with boundary). Let $\rho$ be a unitary operator on $L^2(M,S)$, which preserves $\dom\calD$ and satisfies the assumption below.

\begin{assumption}\label{A:spf}
\begin{enumerate}
\item The commutator $[\calD,\rho]$ is a bounded zeroth-order differential operator (viewed as from $\dom\calD$ to $L^2(M,S)$).

\item For each $r\in[0,1]$, the operator
\[
\calD(r):=(1-r)\calD+r\rho^{-1}\calD\rho=\calD+r\rho^{-1}[\calD,\rho]
\]
is coercive at infinity.
\end{enumerate}
\end{assumption}

Under this assumption, $\calD(r)$, $r\in[0,1]$ is a continuous path of self-adjoint Fredholm Dirac-type operators in the graph topology; cf. \cite{BoossLeschPhillips05,Lesch05sf}. We recall the general definition of spectral flow; cf. \cite{GorokhovskyLesch15}.

\begin{definition}\label{D:spf}
For a graph continuous path of self-adjoint Fredholm operators $f(r)$, $r\in[0,1]$, its \emph{spectral flow} is defined as
\[
\spf(f):=\sum_{j=1}^n\left(\rank\big(1_{[0,\epsilon_j)}(f(r_j))\big)-\rank\big(1_{[0,\epsilon_j)}(f(r_{j-1}))\big)\right),
\]
where $0=r_0<r_1<\dots<r_n=1$ is a subdivision of $[0,1]$ such that there exist $\epsilon_j>0$, $j=1,\dots,n$ satisfying $\pm\epsilon_j\notin\spec f(r)$ and $[-\epsilon_j,\epsilon_j]\cap\spec_{\rm ess}f(r)=\emptyset$ for $r\in[r_{j-1},r_j]$, and $1_{[0,\epsilon)}$ denotes the characteristic function of $[0,\epsilon)$.
\end{definition}

This definition coincides with the intuitive definition that $\spf(f)$ is the net number of eigenvalues of $f(r)$ that change from negative to non-negative as $r$ varies from 0 to 1.
In the case that $f(r)=\calD(r)=(1-r)\calD+r\rho^{-1}\calD\rho$, we denote the spectral flow by $\spf(\calD,\rho)$.

Notice that when $M$ is a manifold with boundary, we can choose a fixed compatible adapted operator $\calA$ on $\p M$ to $\calD(r)$ for any $r\in[0,1]$. Now we focus on the situation of last subsection that $M$ is without boundary and is partitioned along a closed hypersurface $\Sigma$. In this case, we have the following splitting formula for the spectral flow.

\begin{theorem}\label{T:splitting spf}
Let $M^\cut$ be constructed from cutting $M$ along $\Sigma$ as in last subsection. Pick a $Q\in{\rm Gr}(\calA_\Sigma)$ and form $P\in{\rm Gr}(\calA)$ like \eqref{E:s.a-b.c}. Assume $\rho$ is a unitary operator on $M$ that commutes with $Q$ and satisfies Assumption~\ref{A:spf}. Then
\[
\spf(\calD,\rho)=\spf(\calD^\cut_P,\rho).
\]

In particular, if $M$ is decomposed into two components, that is, $M^\cut=M'\cup_\Sigma M''$, then
\[
\spf(\calD,\rho)=\spf(\calD'_Q,\rho')+\spf(\calD''_{\id-Q},\rho''),
\]
where $\calD'$ and $\calD''$ are the restrictions of $\calD$ to $M'$ and $M''$, respectively, similar for $\rho'$ and $\rho''$.
\end{theorem}

\begin{proof}
Since $\rho$ commutes with $Q$, it also commutes with $P(s)$ of \eqref{E:s.a-b.c path} for any $s\in[0,\pi/4]$. Let $\calD_s$, $s\in[0,\pi/4]$ be the family of self-adjoint Fredholm operators on $M^\cut$ associated to $P(s)$ considered in Lemma~\ref{L:path s.a-b.c}. Then $\rho$ preserves $\dom\calD_s$ for any $s\in[0,\pi/4]$.

From Lemma~\ref{L:path s.a-b.c}, viewed as a family of self-adjoint Fredholm operators from $\dom\calD_0$ to $L^2(M^\cut,S_{|M^\cut})$, $\calD_s$, $s\in[0,\pi/4]$ is continuous in the norm topology. It is then graph continuous by \cite[Proposition~2.2]{Lesch05sf}. Now by the homotopy invariance of spectral flow and argue as in the proof of \cite[Proposition~2.1]{GorokhovskyLesch15}, $\spf(\calD_0,\rho)=\spf(\calD_{\pi/4},\rho)$. Note that $\calD_0=\calD^\cut_P$, $\calD_{\pi/4}=\calD$, the thesis then follows.
\end{proof}

\begin{remark}\label{R:conjugate Dirac}
If $\rho\in C^\infty(M,U(k))$, namely $\rho$ is a smooth function on $M$ with values in the unitary group $U(k)$, consider the twisted bundle $S\otimes\CC^k$. One can check that it is again a Dirac bundle with Clifford multiplication acting as identity on $\CC^k$ and with a family of connections given by
\[
\nabla^{S\otimes\CC^k}(r):=\nabla^S\otimes\id+\id\otimes(\D+r\rho^{-1}[\D,\rho]),\quad 0\le r\le1.
\]
Compare \cite[Section~3.1]{LiSuWang24}. This induces a family of Dirac operators that is exactly the family $\calD(r)=(1-r)\calD+r\rho^{-1}\calD\rho$ as above. In other words, we view each $\calD(r)$ as a Dirac operator acting on sections of $S\otimes\CC^k$, although it may not be mentioned explicitly. We will mainly focus on this situation in the discussion below.
\end{remark}

\begin{remark}\label{R:splitting spf}
Suppose $Q$ defines a local boundary condition and $\rho\in C^\infty(M,U(k))$ as above. One sees that $\rho$ commutes with $Q$, thus preserves $\dom(\calD'_Q)$ and $\dom(\calD''_{\id-Q})$. If furthermore, $\rho$ is locally constant at infinity, then the hypothesis of Theorem \ref{T:splitting spf} is satisfied, and the splitting formula holds.
\end{remark}

\section{Callias operators in odd dimensions from a Gromov--Lawson pair}\label{S:Callias}

In this section, we study Callias operators on manifolds with boundary, following the relative Dirac bundle set-up of Cecchini--Zeidler \cite{CeccZeid24GT}. For the Callias operators constructed from a Gromov--Lawson pair, we get a formula for the spectral flow, which is an odd-dimensional analogue of \cite[Corollary~3.9]{CeccZeid24GT}.

\subsection{Relative Dirac bundle and Callias operators on manifolds with boundary}\label{SS:Callias}

We recall Cecchini--Zeidler's definition of relative Dirac bundle, which provides an abstract setting for the construction of Callias operators.

\begin{definition}[{\cite[Definition~2.2]{CeccZeid24GT}}]\label{D:rel Dirac}
Let $M$ be a complete Riemannian manifold (possibly with compact boundary) endowed with a Dirac bundle $S$. Let $K\Subset M^\circ$ be a compact subset in the interior of $M$. $S$ is said to be a \emph{relative Dirac bundle} with support $K$ if there is a self-adjoint, parallel bundle involution $\theta\in C^\infty(M\setminus K,\End(S))$ such that $\rmc(\xi)\theta=-\theta\rmc(\xi)$ for any $\xi\in T^*M_{|M\setminus K}$ and $\theta$ admits a smooth extension to a bundle endomorphism on an open neighborhood of $\overline{M\setminus K}$.
\end{definition}

Let $(S,\theta)$ be a relative Dirac bundle with support $K\Subset M^\circ$. Consider a \emph{Lipschitz} function $\psi:M\to\RR$ such that $\psi=0$ on $K$. Extending $\psi\theta$ by zero on $K$, one can construct a formally self-adjoint Dirac-type operator
\begin{equation}\label{E:Callias}
\calD_\psi:=D+\psi\theta,
\end{equation}
where $D$ is the Dirac operator on $S$. Note that $\theta$ anti-commutes with $D$, so
\[
\calD_\psi^2=D^2+\psi^2+\rmc(\D\psi)\theta.
\]
Assume that $\psi^2-|\D\psi|$ is uniformly positive outside a compact subset. In this case $\calD_\psi$ is called a \emph{Callias operator} with potential $\psi$. This notion coincides with the traditional concept of Callias-type operators considered before (see, e.g., \cite{Callias78,Anghel93Callias,Bunke95}).
In this paper, we shall focus on a special kind of potentials.

\begin{definition}[{\cite[Definition~3.1]{CeccZeid24GT}}]\label{D:admi potential}
If there exists a compact subset $K\Subset L\Subset M$ with $\psi$ equal to a non-zero constant on each component of $M\setminus L$, then $\psi$ is called an \emph{admissible potential}.
\end{definition}

Having a relative Dirac bundle $(S,\theta)$ and an admissible potential $\psi$ on a complete manifold $M$ with compact boundary, one can define a local boundary condition as follows. Let $\bfs:\p M\to\{\pm1\}$ be a locally constant function. The \emph{boundary chirality} associated to $\bfs$ is defined to be
\[
\chi:=\bfs\rmc(\nu^*)\theta:S_{|\p M}\to S_{|\p M}.
\]
Note that $\chi$ is a self-adjoint involution. Thus, it induces an orthogonal decomposition $S_{|\p M}=S^+\oplus S^-$, where $S^\pm$ are the $\pm1$-eigenspaces of $\chi$. Since $\chi$ anti-commutes with $\rmc(\nu^*)$ while commutes with $\rmc(\xi)$ for $\xi\in T^*\p M$, it anti-commutes with $\rmc(\nu^*)^{-1}\rmc(\xi)$. This means that $\rmc(\nu^*)$ interchanges $S^+$ and $S^-$, and that $\chi$ anti-commutes with $\calA$ (a compatible adapted operator to $\calD_\psi$). By \cite[Example~4.20]{BaerBallmann16}, $\chi$ induces a self-adjoint elliptic local boundary condition for $\calD_\psi$ with the domain given by
\[
H^1_{\theta,\bfs}(M,S):=\left\{u\in H^1_D(M,S)\;|\;\chi(u_{|\p M})=u_{|\p M}\right\}.
\]
(Here $H^1_D(M,S)=H^1_{\calD_\psi}(M,S)$ as $\psi\theta\in L^\infty(M,\End(S))$.)
In the perspective of Section~\ref{SS:s.a-Fred Dirac}, this condition is induced by the orthogonal projection $P:=\frac{1}{2}(\id-\chi)\in{\rm Gr}(\calA)$. Since a Callias operator is coercive at infinity, by Proposition~\ref{P:s.a-Fredholm}, $\calD_\psi$ with the above boundary condition, denoted by $\calD_{\psi,\bfs}$, is a self-adjoint Fredholm operator. (When the potential $\psi$ is smooth, this is standard. It is not hard to extend it to Lipschitz situation (see \cite[Section~3]{CeccZeid24GT}).)

\subsection{Spectral flow of Callias operators}\label{SS:Callias spf}

In last subsection, we define a Callias operator $\calD_{\psi,\bfs}$ on a complete Riemannian manifold $M$ with compact boundary, which is self-adjoint and Fredholm. Now we discuss its spectral flow.

Let $\rho$ be a smooth function on $M$ with values in the unitary group $U(k)$ such that $\rho$ is locally constant at infinity. Then $\calD_{\psi,\bfs}$ and $\rho$ satisfy Assumption~\ref{A:spf}, so one can talk about the spectral flow $\spf(\calD_{\psi,\bfs},\rho)$. From \eqref{E:Callias}, it corresponds to the family of operators
\begin{equation}\label{E:Dpsi(r)}
\calD_\psi(r):=(1-r)\calD_\psi+r\rho^{-1}\calD_\psi\rho=D(r)+\psi\theta,\quad 0\le r\le1
\end{equation}
with the chiral boundary condition. Here $D(r)=(1-r)D+r\rho^{-1}D\rho$ is a family of Dirac operators due to Remark~\ref{R:conjugate Dirac}. By the fact that $\rho$ commutes with both $\psi$ and $\theta$, one computes from the BSLW formula \eqref{E:BSLW for} that
\begin{eqnarray}
\calD_\psi^2(r)&\!\!\!=\!\!\!&D^2(r)+\psi^2+\rmc(\D\psi)\theta \label{E:Dpsi(r) square}\\
&\!\!\!=\!\!\!&(\nabla^*\nabla)(r)+\calR(r)+\psi^2+\rmc(\D\psi)\theta.\label{E:Dpsi(r) square-1}
\end{eqnarray}

\begin{lemma}\label{L:spf homotop inv}
Let $\psi_1,\psi_2$ be two admissible potentials on $M$ that coincide at infinity. Then for any $\rho\in C^\infty(M,U(k))$ that is locally constant at infinity and a choice of signs $\bfs:\pM\to\{\pm1\}$,
\[
\spf(\calD_{\psi_1,\bfs},\rho)=\spf(\calD_{\psi_2,\bfs},\rho).
\]
\end{lemma}

\begin{proof}
This lemma is an immediate consequence of the homotopy invariance of the spectral flow \cite[Proposition~2.1]{GorokhovskyLesch15} after noticing that $\calD_{\psi_2}-\calD_{\psi_1}$ is a compact operator from $H^1_{\theta,\bfs}(M,S)$ to $L^2(M,S)$ (cf. \cite[Lemma~3.3]{CeccZeid24GT}), so that one can connect them by a continuous path of self-adjoint Fredholm operators.
\end{proof}

The following two lemmas are spectral flow versions of \cite[Lemmas~3.7 and 3.8]{CeccZeid24GT} with essentially the same proofs (using \eqref{E:Dpsi(r) square}).

\begin{lemma}\label{L:spf vanishing-1}
Let $M$ be a complete Riemannian manifold with compact boundary, endowed with a relative Dirac bundle $(S,\theta)$. Let $\psi$ be an admissible potential on $M$ and $\bfs:\pM\to\{\pm1\}$ be a choice of signs such that
\begin{enumerate}
\item there exists $C>0$ such that $\psi^2-|\D\psi|\ge C$ on all of $M$; and

\item $\bfs\psi\ge0$ along $\pM$.
\end{enumerate}
Then for all $r\in[0,1]$ and $\rho\in C^\infty(M,U(k))$, which is locally constant at infinity, the operator $\calD_{\psi,\bfs}(r)$ is invertible. In particular, $\spf(\calD_{\psi,\bfs},\rho)=0$.
\end{lemma}

\begin{lemma}\label{L:spf vanishing-2}
Let $M$ be a compact Riemannian manifold with boundary and $(S,\theta)$ be a relative Dirac bundle over $M$ with \emph{empty} support. Let $\bfs:\pM\to\{\pm1\}$ be a choice of signs. Then any Lipschitz function $\psi:M\to\RR$ is an admissible potential and for any $\rho\in C^\infty(M,U(k))$, $\spf(\calD_{\psi,\bfs},\rho)$ is independent of $\psi$.

Furthermore, if the sign $\bfs$ is constant on all of $\pM$, then $\spf(\calD_{\psi,\bfs},\rho)=0$ for any $\psi$ and $\rho$.
\end{lemma}

Now consider the case that $M$ is without boundary, and $M^\cut$ is obtained from cutting $M$ along a closed hypersurface $\Sigma$ as in Section~\ref{SS:partition}. Suppose the choice of signs $\bfs:\p M^\cut=\Sigma_1\sqcup\Sigma_2\to\{\pm1\}$ satisfies $\bfs_{|\Sigma_1}=\bfs_{|\Sigma_2}$. Then on $\Sigma_1$ the boundary condition is induced by $Q:=\frac{1}{2}(\id-\chi_{|\Sigma_1})$, while on $\Sigma_2$ the boundary condition is induced by $\frac{1}{2}(\id-\chi_{|\Sigma_2})=\frac{1}{2}(\id+\chi_{|\Sigma_1})=\id-Q$. (This is because $\rmc(\nu^*_{|\Sigma_1})=-\rmc(\nu^*_{|\Sigma_2})$.) Therefore, in this case we get a self-adjoint elliptic boundary condition. It then follows from Theorem~\ref{T:splitting spf} and Remark~\ref{R:splitting spf} that

\begin{corollary}\label{C:spf Callias}
Let $M^\cut$ be constructed from cutting $M$ along $\Sigma$ with $\p M^\cut=\Sigma_1\sqcup\Sigma_2$ as above. Let $\calD_\psi$ be a Callias operator on $M$ and $\calD^\cut_\psi$ be the resulting operator on $M^\cut$. Choose signs $\bfs:\p M^\cut\to\{\pm1\}$ such that $\bfs_{|\Sigma_1}=\bfs_{|\Sigma_2}$. Then for any $\rho\in C^\infty(M,U(k))$ that is locally constant at infinity,
\[
\spf(\calD_\psi,\rho)=\spf(\calD^\cut_{\psi,\bfs},\rho).
\]
\end{corollary}

The following spectral estimate for $\calD_{\psi,\bfs}(r)$ from \cite[Section~4]{CeccZeid24GT} is a consequence of \eqref{E:Dpsi(r) square-1}, \eqref{E:Penrose} and Green's formula.

\begin{proposition}[{\cite[Theorem~4.3]{CeccZeid24GT}}]\label{P:sp est Callias}
Let $\calD_{\psi,\bfs}(r)$, $r\in[0,1]$, be the Callias operator \eqref{E:Dpsi(r)} associated to an admissible potential $\psi$, a choice of signs $\bfs:\pM\to\{\pm\}$, and a function $\rho\in C^\infty(M,U(k))$ on a compact Riemannian $n$-manifold $M$. Then for any $u\in\dom(\calD_{\psi,\bfs}(r))=H^1_{\theta,\bfs}(M,S)$, there holds the estimate
\[
\begin{aligned}
\int_M|\calD_\psi(r)u|^2 \D V_M &=\frac{n}{n-1} \int_M(|\calP(r)u|^2+\langle u,\calR(r)u\rangle)\D V_M \\
  &\quad +\int_M\langle u,(\psi^2 + \rmc(\D\psi)\theta)u\rangle \D V_M+\int_{\pM}\left(\frac{1}{2}n H+\bfs\psi \right)|\tau(u)|^2 \D V_{\pM} \\
  &\ge\frac{n}{n-1}\int_M\langle u,\calR(r)u\rangle\ \D V_M+\int_M \left(\psi^2-|\D\psi|\right)|u|^2 \D V_M \\
  &\quad +\int_{\pM}\left(\frac{1}{2}n H+\bfs\psi \right)|\tau(u)|^2 dV_{\pM},
\end{aligned}
\]
where $\calR(r)$ is the curvature endomorphism \eqref{E:BSLW endo} in the BSLW formula of $D^2(r)$, $H$ is the mean curvature of $\pM$, and $\tau(u):=u_{|\pM}$.

Moreover, the equality holds if and only if
\[
\begin{aligned}
\calP_e(r)u=\nabla_e(r)u+\frac{1}{n}\rmc(e^*)D(r)u &=0,\quad\text{for all }e\in TM, \\
\left(\rmc(\D\psi)\theta+|\D\psi|\right)u &=0.
\end{aligned}
\]
\end{proposition}

\begin{remark}\label{R:sp est Callias}
We can also consider the case that $M$ is a complete (not necessarily compact) Riemannian manifold without boundary and $\rho$ is locally constant at infinity. In this case, $\calD_\psi(r)u$ is $L^2$-integrable for any $u\in\dom(\calD_\psi(r))=H^1_D(M,S)$ and it holds that
\[
\int_M|\calD_\psi(r)u|^2 \D V_M=\int_M\Big(\frac{n}{n-1}|\calP(r)u|^2+\langle u,\Theta(r)u\rangle\Big)\D V_M,
\]
with $\Theta(r)$ a bundle endomorphism given by
\[
\begin{aligned}
\Theta(r):=&\,\frac{n}{n-1}\calR(r)+\psi^2+\rmc(\D\psi)\theta \\
\ge&\,\frac{n}{n-1}\calR(r)+\psi^2-|\D\psi|.
\end{aligned}
\]
If the bundle endomorphism $\frac{n}{n-1}\calR(r)+\psi^2-|\D\psi|$ is a non-negative operator, then $\calP(r)u$ is $L^2$-integrable, so that $\int_M\langle u,\Theta(r)u\rangle\D V_M$ converges absolutely. Therefore, we still have the following similar estimate
\[
\begin{aligned}
\int_M|\calD_\psi(r)u|^2 \D V_M &=\frac{n}{n-1} \int_M|\calP(r)u|^2\D V_M+\int_M\langle u,\Theta(r)u\rangle\D V_M \\
  &\ge\int_M\Big\langle u,\Big(\frac{n}{n-1}\calR(r)+\psi^2-|\D\psi|\Big)u\Big\rangle \D V_M.
\end{aligned}
\]
Compare \cite[Section~2]{GromovLawson83}. The conditions for the equality to hold are the same as Proposition~\ref{P:sp est Callias}.
\end{remark}

\subsection{Callias operators from a Gromov--Lawson pair}\label{SS:Callias GL pair}

In this subsection, we study a special case of relative Dirac bundle and Callias operators on an odd-dimensional spin manifold.

Let $(M,g)$ be an odd-dimensional compact Riemannian spin manifold with boundary, and let $\slaS$ be the complex spinor bundle over $M$. If $E,F\to M$ are two Hermitian bundles with Hermitian connections, then the twisted bundle
\begin{equation}\label{E:twist bundle}
S:=\slaS\otimes(E\oplus F)=(\slaS\otimes E)\oplus(\slaS\otimes F)
\end{equation}
is a Dirac bundle over $M$, with the connection being the usual tensor product connection and the Clifford multiplication given by
\[
\rmc(\xi)=\left(
\begin{matrix}
\rmc'(\xi)\otimes\id_E & 0 \\
0 & -\rmc'(\xi)\otimes\id_F
\end{matrix}\right),\quad\xi\in T^*M,
\]
where $\rmc'(\xi)$ is the Clifford multiplication on $\slaS$. The associated Dirac operator becomes
\begin{equation}\label{E:twist Dirac}
D=\left(
\begin{matrix}
\slaD_E & 0 \\
0 & -\slaD_F
\end{matrix}\right),
\end{equation}
where $\slaD_E,\slaD_F$ are the twisted spin Dirac operators. We call the operator $D$ twisted by $E\oplus F^{\rm op}$. For the operator $D$, the curvature endomorphism \eqref{E:BSLW endo} in the BSLW formula is
\[
\calR=\frac{1}{4}\scal_g+\calR^{E\oplus F},
\]
where $\calR^{E\oplus F}=\calR^E\oplus\calR^F$, with $\calR^E$ (resp. $\calR^F$) being the curvature endomorphism of the curvature tensor $(\nabla^E)^2$ (resp. $(\nabla^F)^2$).

Now we make the following assumption (cf. \cite[(2-15)]{CeccZeid24GT}) on $E,F$.

\begin{assumption}\label{A:GL pair}
$(E,F)$ is a \emph{Gromov--Lawson pair} with support $K$, that is, there exist a compact subset $K\Subset M^\circ$ and a parallel unitary bundle isomorphism $\frakt:E_{|M\setminus K}\to F_{|M\setminus K}$ that extends to a smooth bundle endomorphism on a neighborhood of $\overline{M\setminus K}$.
\end{assumption}

Under this assumption, $S$ becomes a relative Dirac bundle with involution
\begin{equation}\label{E:GL pair rel Dirac}
\theta=\left(
\begin{matrix}
0 & \id\otimes\frakt^* \\
\id\otimes\frakt & 0
\end{matrix}\right):S_{|M\setminus K}\to S_{|M\setminus K}.
\end{equation}
So we can talk about the Callias operator $\calD_\psi$ \eqref{E:Callias} in this case.

One can use a Gromov--Lawson pair $(E,F)$ over $M$ to construct a Hermitian bundle over the closed double $\double M=M\cup_{\pM}M^-$, where $M^-$ is $M$ with opposite orientation. To be precise, using the bundle isomorphism $\frakt$, we define a bundle $V(E,F)$ over $\double M$ such that it coincides with $E$ over $M$ and with $F$ over $M^-$ outside a small collar neighborhood of $\pM$. Let $\slaD_{V(E,F)}$ denote the spin Dirac operator over the spin manifold $\double M$ twisted by $V(E,F)$.

Let $\rho=\rho^+\oplus\rho^-\in C^\infty(M,U(k)\oplus U(k))$. Suppose in a collar neighborhood of $\pM$, $\rho^+=\rho^-$ is locally constant in the normal direction. Then $\rho^+$ and $\rho^-$ can be glued smoothly to yield $\tilde{\rho}\in C^\infty(\double M,U(k))$, which means that $\tilde{\rho}$ is an extension of $\rho^+$ to $\double M$ such that $\tilde{\rho}_{|M^-}$ coincides with $\rho^-$. Consider the Callias operator $\calD_{\psi,1}$ on $M$ with the boundary condition given by choosing the sign $\bfs=1$.

\begin{center}
\tikzset{every picture/.style={line width=0.8pt}} 
\begin{tikzpicture}[x=0.75pt,y=0.75pt,yscale=-1,xscale=1]

\draw   (113,128.91) .. controls (113,120.03) and (126.61,112.82) .. (143.39,112.82) .. controls (160.18,112.82) and (173.79,120.03) .. (173.79,128.91) .. controls (173.79,137.8) and (160.18,145) .. (143.39,145) .. controls (126.61,145) and (113,137.8) .. (113,128.91) -- cycle ;
\draw   (113,171.09) .. controls (113,162.2) and (126.61,155) .. (143.39,155) .. controls (160.18,155) and (173.79,162.2) .. (173.79,171.09) .. controls (173.79,179.97) and (160.18,187.18) .. (143.39,187.18) .. controls (126.61,187.18) and (113,179.97) .. (113,171.09) -- cycle ;
\draw    (113,128.91) .. controls (113.47,68.84) and (173.33,44.24) .. (198,54.24) .. controls (222.67,64.24) and (212.93,59.44) .. (233,54.24) .. controls (253.07,49.04) and (267.87,99.8) .. (268,130) ;
\draw   (228.92,130) .. controls (228.92,124.48) and (237.67,120) .. (248.46,120) .. controls (259.25,120) and (268,124.48) .. (268,130) .. controls (268,135.52) and (259.25,140) .. (248.46,140) .. controls (237.67,140) and (228.92,135.52) .. (228.92,130) -- cycle ;
\draw    (113,171.09) .. controls (113.07,231.09) and (162.93,254.4) .. (193,245) .. controls (223.07,235.6) and (213.33,240) .. (233,245) .. controls (252.67,250) and (267.87,197.21) .. (268,172.01) ;
\draw   (228.92,172.01) .. controls (228.92,166.49) and (237.67,162.01) .. (248.46,162.01) .. controls (259.25,162.01) and (268,166.49) .. (268,172.01) .. controls (268,177.54) and (259.25,182.01) .. (248.46,182.01) .. controls (237.67,182.01) and (228.92,177.54) .. (228.92,172.01) -- cycle ;
\draw    (173.79,128.91) .. controls (173.47,94.84) and (233.07,100.6) .. (228.92,130) ;
\draw    (173.79,171.09) .. controls (173.47,204.6) and (233.07,202.41) .. (228.92,172.01) ;
\draw    (158,84.24) .. controls (168.27,100.04) and (193.07,94.44) .. (198,79.24) ;
\draw    (158,215) .. controls (168.27,200.6) and (192.67,204.6) .. (198,220) ;
\draw    (163,89.24) .. controls (168.27,79.24) and (182.67,69.64) .. (195,84.24) ;
\draw    (163,210) .. controls (168.27,220.2) and (183.47,230.2) .. (195,215) ;
\draw    (365,150.97) .. controls (365.47,90.9) and (425.33,66.3) .. (450,76.3) .. controls (474.67,86.3) and (464.93,81.5) .. (485,76.3) .. controls (505.07,71.1) and (519.87,121.86) .. (520,152.06) ;
\draw    (365,149.07) .. controls (365.07,209.07) and (414.93,232.39) .. (445,222.99) .. controls (475.07,213.59) and (465.33,217.99) .. (485,222.99) .. controls (504.67,227.99) and (519.87,175.2) .. (520,150) ;
\draw    (425.79,150.97) .. controls (425.47,116.9) and (485.07,122.66) .. (480.92,152.06) ;
\draw    (425.79,149.07) .. controls (425.47,182.59) and (485.07,180.4) .. (480.92,150) ;
\draw    (410,106.3) .. controls (420.27,122.1) and (445.07,116.5) .. (450,101.3) ;
\draw    (410,192.99) .. controls (420.27,178.59) and (444.67,182.59) .. (450,197.99) ;
\draw    (415,111.3) .. controls (420.27,101.3) and (434.67,91.7) .. (447,106.3) ;
\draw    (415,187.99) .. controls (420.27,198.19) and (435.47,208.19) .. (447,192.99) ;
\draw  [draw opacity=0] (425.79,149.07) .. controls (425.79,149.07) and (425.79,149.07) .. (425.79,149.07) .. controls (425.79,149.07) and (425.79,149.07) .. (425.79,149.07) .. controls (425.79,158.15) and (412.33,165.51) .. (395.73,165.51) .. controls (379.12,165.51) and (365.66,158.15) .. (365.66,149.07) -- (395.73,149.07) -- cycle ; \draw   (425.79,149.07) .. controls (425.79,149.07) and (425.79,149.07) .. (425.79,149.07) .. controls (425.79,149.07) and (425.79,149.07) .. (425.79,149.07) .. controls (425.79,158.15) and (412.33,165.51) .. (395.73,165.51) .. controls (379.12,165.51) and (365.66,158.15) .. (365.66,149.07) ;  
\draw  [draw opacity=0][dash pattern={on 2.25pt off 2.25pt on 2.25pt off 2.25pt}] (365.66,149.07) .. controls (365.66,149.07) and (365.66,149.07) .. (365.66,149.07) .. controls (365.66,149.07) and (365.66,149.07) .. (365.66,149.07) .. controls (365.66,139.99) and (379.12,132.63) .. (395.73,132.63) .. controls (412.33,132.63) and (425.79,139.99) .. (425.79,149.07) -- (395.73,149.07) -- cycle ; \draw  [dash pattern={on 2.25pt off 2.25pt on 2.25pt off 2.25pt}] (365.66,149.07) .. controls (365.66,149.07) and (365.66,149.07) .. (365.66,149.07) .. controls (365.66,149.07) and (365.66,149.07) .. (365.66,149.07) .. controls (365.66,139.99) and (379.12,132.63) .. (395.73,132.63) .. controls (412.33,132.63) and (425.79,139.99) .. (425.79,149.07) ;  
\draw  [draw opacity=0] (520,150) .. controls (520,150) and (520,150) .. (520,150) .. controls (520,150) and (520,150) .. (520,150) .. controls (520,156.35) and (511.25,161.5) .. (500.46,161.5) .. controls (489.67,161.5) and (480.92,156.35) .. (480.92,150) -- (500.46,150) -- cycle ; \draw   (520,150) .. controls (520,150) and (520,150) .. (520,150) .. controls (520,150) and (520,150) .. (520,150) .. controls (520,156.35) and (511.25,161.5) .. (500.46,161.5) .. controls (489.67,161.5) and (480.92,156.35) .. (480.92,150) ;  
\draw  [draw opacity=0][dash pattern={on 2.25pt off 2.25pt on 2.25pt off 2.25pt}] (480.92,150) .. controls (480.92,150) and (480.92,150) .. (480.92,150) .. controls (480.92,150) and (480.92,150) .. (480.92,150) .. controls (480.92,143.65) and (489.67,138.5) .. (500.46,138.5) .. controls (511.25,138.5) and (520,143.65) .. (520,150) -- (500.46,150) -- cycle ; \draw  [dash pattern={on 2.25pt off 2.25pt on 2.25pt off 2.25pt}] (480.92,150) .. controls (480.92,150) and (480.92,150) .. (480.92,150) .. controls (480.92,150) and (480.92,150) .. (480.92,150) .. controls (480.92,143.65) and (489.67,138.5) .. (500.46,138.5) .. controls (511.25,138.5) and (520,143.65) .. (520,150) ;  
\draw    (295,150) -- (338,150) ;
\draw [shift={(340,150)}, rotate = 180] [color={rgb, 255:red, 0; green, 0; blue, 0 }  ][line width=0.75]    (10.93,-3.29) .. controls (6.95,-1.4) and (3.31,-0.3) .. (0,0) .. controls (3.31,0.3) and (6.95,1.4) .. (10.93,3.29)   ;
\draw  [dash pattern={on 2.25pt off 2.25pt on 2.25pt off 2.25pt}]  (138,72) .. controls (139.47,83.2) and (134.67,88.4) .. (153,95) .. controls (171.33,101.6) and (161.47,117.2) .. (183,110) ;
\draw  [dash pattern={on 2.25pt off 2.25pt on 2.25pt off 2.25pt}]  (223,58) .. controls (229.07,68.8) and (221.07,77.2) .. (233,85) .. controls (244.93,92.8) and (237.07,101.2) .. (218,110) ;

\draw (191,32.4) node [anchor=north west][inner sep=0.75pt]  {$M$};
\draw (184,252.4) node [anchor=north west][inner sep=0.75pt]  {$M^{-}$};
\draw (434,252.4) node [anchor=north west][inner sep=0.75pt]  {$\mathsf{D} M$};
\draw (204,72.4) node [anchor=north west][inner sep=0.75pt]    {$K$};
\draw (271,102.4) node [anchor=north west][inner sep=0.75pt]    {$\rho ^{+} \oplus \rho ^{-}$};
\draw (271,202.4) node [anchor=north west][inner sep=0.75pt]    {$\rho ^{-} \oplus \rho ^{-}$};
\draw (525,150.4) node [anchor=north west][inner sep=0.75pt]    {$\tilde{\rho } \oplus \tilde{\rho } '$};
\draw (270,82.4) node [anchor=north west][inner sep=0.75pt]    {$E,F$};
\draw (271,182.4) node [anchor=north west][inner sep=0.75pt]    {$F,F$};
\draw (523,124.4) node [anchor=north west][inner sep=0.75pt]    {$V( E,F) ,\tilde{F}$};
\end{tikzpicture}

\textbf{Fig. 1.} Construction on the closed double.
\end{center}

\begin{theorem}\label{T:GL pair spf}
Under the above setting, for any admissible potential $\psi$,
\[
\spf(\calD_{\psi,1},\rho)=\spf(\slaD_{V(E,F)},\tilde{\rho}).
\]
\end{theorem}

\begin{proof}
Denote $\tilE=V(E,F)$ and let $\tilF$ be the extension of $F$ to $\double M$ such that $\tilF_{|M^-}=F$. Replacing $E$ and $F$ by $\tilE$ and $\tilF$ respectively in \eqref{E:twist bundle} and \eqref{E:twist Dirac}, one gets a twisted bundle $\tilS$ and a Dirac operator $\tilD$ twisted by $\tilE\oplus\tilF^{\rm op}$ over $\double M$. By the splitting of $\tilS$ and $\tilD$, we have
\[
\spf(\tilD,\tilde{\rho}\oplus\tilde{\rho}')=\spf(\slaD_{V(E,F)},\tilde{\rho})-\spf(\slaD_{\tilF},\tilde{\rho}'),
\]
where $\tilde{\rho}'\in C^\infty(\double M,U(k))$ is the smooth gluing of $\rho^-$ and $\rho^-$.
Using \cite[Corollary~2.7]{Getzler93}, the last term can be computed as
\[
\spf(\slaD_{\tilF},\tilde{\rho}')=\sqrt{\frac{\varepsilon}{\pi}}\int_0^1\Tr\big((\tilde{\rho}')^{-1}[\slaD_{\tilF},\tilde{\rho}']e^{-\varepsilon\slaD^2_{\tilF}(r)}\big)\D r,
\]
where $\slaD_{\tilF}(r)=(1-r)\slaD_{\tilF}+r(\tilde{\rho}')^{-1}\slaD_{\tilF}\tilde{\rho}'$. Note that the trace can be written as an integration over $\double M$. It follows from the symmetry of $\tilF$ and $\tilde{\rho}'$ that $\spf(\slaD_{\tilF},\tilde{\rho}')$ vanishes. Hence $\spf(\tilD,\tilde{\rho}\oplus\tilde{\rho}')=\spf(\slaD_{V(E,F)},\tilde{\rho})$.

We now compute $\spf(\tilD,\tilde{\rho}\oplus\tilde{\rho}')$. Cutting $\double M$ along $\pM$, by Corollary~\ref{C:spf Callias},
\[
\spf(\tilD,\tilde{\rho}\oplus\tilde{\rho}')=\spf(\calD_{0,1},\rho)+\spf(\calD^{F,F}_{0,1},\rho')=\spf(\calD_{0,1},\rho),
\]
where $\rho'=\rho^-\oplus\rho^-$, and $\calD^{F,F}_{0,1}$ is the Dirac operator on $M^-$ twisted by $F\oplus F^{\rm op}$, whose spectral flow thus vanishes because of Lemma~\ref{L:spf vanishing-2}. Lastly, we have $\spf(\calD_{0,1},\rho)=\spf(\calD_{\psi,1},\rho)$ due to Lemma~\ref{L:spf homotop inv}. The theorem is proved.
\end{proof}

\section{The odd-dimensional long neck problem}\label{S:long neck}

In this section, we use the spectral flow proof of Llarull's theorem by Li--Su--Wang \cite{LiSuWang24} combined with Cecchini--Zeidler's argument \cite{CeccZeid24GT} to derive a scalar-mean curvature comparison theorem for the long neck problem on odd-dimensional spin manifolds.

\subsection{Construction of a Gromov--Lawson pair}\label{SS:GL pair}

We first recall the construction of \cite{LiSuWang24} in proving Llarull's theorem in odd dimensions. Let $(X,g)$ be a compact Riemannian spin manifold of dimension $n$ ($n\ge3$ odd), and $\Theta:X\to\rmS^n$ be a smooth map. There exist a trivial bundle $E_0$ over $\rmS^n$ and a smooth function $\bar{\rho}$ on $\rmS^n$ with values in a unitary group such that
\[
\nabla(r):=\D+r\bar{\rho}^{-1}[\D,\bar{\rho}],\quad 0\le r\le1
\]
defines a family of Hermitian connections on $E_0$. The pull-back $\Theta^*E_0$ is a Hermitian vector bundle over $X$ with a family of Hermitian connections induced by $\Theta^*\bar{\rho}$. Let $\slaD_{\Theta^*E_0}(r)$ be the spin Dirac operator on $X$ twisted by $\Theta^*E_0$ with connection
\[
\nabla^{\slaS_X\otimes\Theta^*E_0}(r)=\nabla^{\slaS_X}\otimes\id+\id\otimes\Theta^*\nabla(r),\quad 0\le r\le1,
\]
where $\slaS_X$ is the complex spinor bundle over $X$. In this case, the curvature endomorphism \eqref{E:BSLW endo} in the BSLW formula of $\slaD^2_{\Theta^*E_0}(r)$ is given by
\[
\calR(r)=\frac{1}{4}\scal_g+\calR^{\Theta^*E_0}(r),
\]
where $\calR^{\Theta^*E_0}(r)$ is the curvature endomorphism of the curvature tensor $(\Theta^*\nabla(r))^2$. When $X$ is closed, one can consider the spectral flow of $\slaD_{\Theta^*E_0}(r)$, $r\in[0,1]$. In view of Remark~\ref{R:conjugate Dirac}, this spectral flow is just $\spf(\slaD_{\Theta^*E_0}(0),\Theta^*\bar{\rho})$. It is proved in \cite[Sections~3.2 and 3.3]{LiSuWang24} that:
\begin{enumerate}
\item For each $r\in[0,1]$ and each $x\in X$, one has
\[
\calR_x^{\Theta^*E_0}(r)\ge-a(x)\cdot\frac{1}{4}n(n-1),
\]
where $a(x)$ is the area contraction constant of $\Theta$ at $x$. (For a smooth map $\Theta:X\to Y$ between two Riemannian manifolds, the \emph{area contraction constant} at $x\in X$ is defined the be the norm of the induced map $\Theta^*:\wedge^2T_x^*X\to\wedge^2T_{\Theta(x)}^*Y$ on 2-forms.) In addition, the inequality is strict unless $r=\frac{1}{2}$.

\item If $X$ is closed, then the spectral flow is given by
\begin{equation}\label{E:spf=deg}
\spf(\slaD_{\Theta^*E_0}(0),\Theta^*\bar{\rho})=-\deg(\Theta).
\end{equation}
\end{enumerate}

We can now formulate the following odd-dimensional analogue of \cite[Lemma~5.1]{CeccZeid24GT}.

\begin{lemma}\label{L:GL pair}
Let $(M,g)$ be an $n$-dimensional ($n\ge3$ odd) compact Riemannian spin manifold with boundary and $\Phi:M\to\rmS^n$ be a smooth area-decreasing map that is locally constant near $\pM$ and of non-zero degree. Set
\[
l=\dist_{g}(\supp(\D\Phi),\pM)>0.
\]
Then there exist a Gromov--Lawson pair $(E,F)$ and a function $\rho=\rho^+\oplus\rho^-\in C^\infty(M,U(k)\oplus U(k))$ as in Theorem~\ref{T:GL pair spf}, where $k=\rank E=\rank F$, such that
\begin{enumerate}
\item $(E,F)$ has support $K:=\{p\in M\;|\;\dist_g(p,\pM)\ge l\}\supset\supp(\D\Phi)$; \label{item:GL pair 1}

\item $\rho^+$ induces a family of Hermitian connections $\nabla^E(r):=\nabla^E+r(\rho^+)^{-1}[\nabla^E,\rho^+]$, $0\le r\le1$ on $E$, and the same is true with $\rho^+$ replaced by $\rho^-$ and $E$ replaced by $F$; \label{item:GL pair 2}

\item For each $r\in[0,1]$ and each $p\in M$,
\[
\calR^E_p(r)\ge-a(p)\cdot\frac{1}{4}n(n-1),\quad\calR^F_p(r)\equiv0,
\]
and the inequality is strict unless $r=\frac{1}{2}$; and \label{item:GL pair 3}

\item $\spf(\slaD_{V(E,F)},\tilde{\rho})\ne0$, where $V(E,F)\to\double M$ is defined in Section~\ref{SS:Callias GL pair} and $\tilde{\rho}\in C^\infty(\double M,U(k))$ is the smooth gluing of $\rho^+$ and $\rho^-$. \label{item:GL pair 4}
\end{enumerate}
\end{lemma}

\begin{proof}
Part~1 is literally the same as the proof of \cite[Lemma~5.1.(\romnu2)]{CeccZeid24GT}. Briefly speaking, by the locally constant property of $\Phi$, one can construct a smooth map $\Psi:M\to\rmS^n$ such that $\Phi=\Psi$ on $\overline{M\setminus K}$ (whose image is a set of finite points), and the induced map $\Psi^*$ on 2-forms vanishes. In particular, the degree of $\Psi$ is zero. Then the bundles $E:=\Phi^*E_0$ and $F:=\Psi^*E_0$ form a Gromov--Lawson pair with support $K$, where $E_0$ is the trivial bundle over $\rmS^n$ discussed above.

Set $\rho^+=\Phi^*\bar{\rho}$, $\rho^-=\Psi^*\bar{\rho}$ ($\bar{\rho}$ is the aforementioned function on $\rmS^n$) and note that $F$ is always a flat bundle (since $\Psi^*$ is the zero map on 2-forms). Then parts~2 and 3 follow immediately from the discussion above. Also, since $\Phi=\Psi$ is locally constant near $\pM$, we get that $\rho^+=\rho^-$ is locally constant near $\pM$. Thus they can be glued smoothly as discussed before Theorem~\ref{T:GL pair spf}.

To show part~4, let $\Theta:\double M\to \rmS^n$ be the smooth map defined by $\Theta_{|M}=\Phi$ and $\Theta_{|M^-}=\Psi$. Then $\deg(\Theta)=\deg(\Phi)\ne0$. Note that $\Theta^*E_0=V(E,F)$, $\Theta^*\bar{\rho}=\tilde{\rho}$. It follows from \eqref{E:spf=deg} (with $X$ being $\double M$) that
\[
\spf(\slaD_{V(E,F)},\tilde{\rho})=\spf(\slaD_{\Theta^*E_0}(0),\Theta^*\bar{\rho})=-\deg(\Theta)\ne0.
\]
This completes the proof.
\end{proof}

\subsection{Estimate on the length of the neck}\label{SS:neck length}

In this subsection, we prove Theorem~\ref{IT:long neck}, which is reformulated as follows.

\begin{theorem}\label{T:neck length}
Let $(M,g)$ be an $n$-dimensional ($n\ge3$ odd) connected compact Riemannian spin manifold with boundary. Let $f:M\to\rmS^n$ be a smooth area-decreasing map that is locally constant near $\pM$ and of non-zero degree. Assume that $\scal_g\ge n(n-1)$ on $\supp(\D f)$, $\scal_g\ge\sigma^2n(n-1)$ on $M\setminus\supp(\D f)$ for some $\sigma>0$, and that $H_g\ge-\sigma\tan(\frac{1}{2}\sigma nl)$ for some $l\in(0,\frac{\pi}{\sigma n})$. Then $\dist_g(\supp(\D f),\pM)<l$.
\end{theorem}

\begin{proof}
We prove by contradiction. Suppose $\dist_g(\supp(\D f),\pM)\ge l$. Pick a Gromov--Lawson pair $(E,F)$ with support $K$ that satisfies the conditions of Lemma~\ref{L:GL pair}. Let $D$ be the twisted spin Dirac operator of \eqref{E:twist Dirac}. As in \cite[Section~5]{CeccZeid24GT}, construct the admissible potential $\psi:=h(x)$, where $h(t)=\frac{1}{2}\sigma n\tan(\frac{1}{2}\sigma nt)$ and
\[
x:M\to[0,l],\quad x(p):=\min\{\dist_g(K,p),l\}.
\]
So we get
\[
\psi^2-|\D\psi|\ge h(x)^2-h'(x)=-\frac{1}{4}\sigma^2n^2
\]
almost everywhere on $M$ and
\begin{equation}\label{E:psi > mean cur}
\psi_{|\pM}=\frac{1}{2}\sigma n\tan\Big(\frac{1}{2}\sigma nl\Big)\ge-\frac{1}{2}nH_g.
\end{equation}

Consider the Callias operator $\calD_\psi=D+\psi\theta$, where $\theta$ is given in \eqref{E:GL pair rel Dirac}. Impose the chiral boundary condition with the choice of signs $\bfs=1$. By Theorem~\ref{T:GL pair spf} and Lemma~\ref{L:GL pair}, there exists a function $\rho=\rho^+\oplus\rho^-\in C^\infty(M,U(k)\oplus U(k))$ such that
\[
\spf(\calD_{\psi,1},\rho)\ne0.
\]
This means that there exist some $r\in[0,1]$ and $0\ne u\in\ker(\calD_{\psi,1}(r))$ (as an $L^2$-section). However, using the spectral estimate Proposition~\ref{P:sp est Callias} and repeating the argument of \cite[Proof of Theorem~1.4]{CeccZeid24GT}, we can deduce that this is impossible. Roughly speaking, by Proposition~\ref{P:sp est Callias} and \eqref{E:psi > mean cur}, one has
\[
0\ge\frac{n}{n-1} \int_M\Big(\frac{1}{4}\scal_g|u|^2+\langle u,\calR^{E\oplus F}(r)u\rangle\Big)\D V_M+\int_M\langle u,(\psi^2-|\D\psi|)u\rangle \D V_M.
\]
From this, Lemma~\ref{L:GL pair}.3 and the hypothesis on the lower bound of the scalar curvature, one gets that $u=0$ on some non-empty open subset of $\supp(\D f)$ (where the area contraction constant $a(p)$ is strictly less than 1) and at the same time satisfies the equality condition of Proposition~\ref{P:sp est Callias}. It then follows from \cite[Remark~4.5]{CeccZeid24GT} that $u$ vanishes almost everywhere on $M$---a contradiction! And the theorem is proved.
\end{proof}

\section{A quantitative Llarull's theorem on non-compact manifolds}\label{S:quanti Llarull}

In this section, we employ ideas from last section to prove the quantitative Llarull's theorem (Theorem~\ref{IT:quanti Llarull}) on complete non-compact spin manifolds. The point is to apply the index theory of Callias operators from a Gromov--Lawson pair with the potential given by a distance-related function. 

\begin{proof}[Proof of Theorem~\ref{IT:quanti Llarull}]
Again prove by contradiction and suppose
\begin{equation}\label{E:scal ge tan^2}
\scal_g\ge-\sigma^2n(n-1)\tan^2\Big(\frac{1}{2}\sigma n\delta\Big)\mbox{ on }M.
\end{equation}
For the reader's convenience, we repeat the hypothesis on the lower bound of the scalar curvature as follows:
\begin{eqnarray}
&\scal_g\ge n(n-1)\mbox{ on }\supp(\D f),\label{E:scal bound supp}\\ 
&\scal_g\ge\sigma^2n(n-1)\mbox{ on }K\setminus\supp(\D f). \label{E:scal bound K}
\end{eqnarray}

We mainly demonstrate the proof for $n$ odd.
Let $L\Subset M$ be another compact subset with smooth boundary such that $K\subset L^\circ$. As in Lemma~\ref{L:GL pair}, we can construct a Gromov--Lawson pair $(E,F)$ on $L$ with support $\supp(\D f)$ and a function $\rho=\rho^+\oplus\rho^-\in C^\infty(L,U(k)\oplus U(k))$ that satisfy conditions~1--4 of Lemma~\ref{L:GL pair}. By the triviality of $(E,F)$ and the fact that $\rho$ is locally constant outside $\supp(\D f)$, they can be extended trivially to the whole of $M$.

Like in last section, construct an admissible potential $\psi:=h(x)$, where $h(t)=\frac{1}{2}\sigma n\tan(\frac{1}{2}\sigma nt)$ and
\[
x:M\to[0,\delta],\quad x(p):=\min\{\dist_g(\supp(\D f),p),\delta\}.
\]
Then
\begin{equation}\label{E:psi^2-dpsi bound}
\psi^2-|\D\psi|\ge-\frac{1}{4}\sigma^2n^2
\end{equation}
almost everywhere on $M$ and
\begin{equation}\label{E:psi on M-K}
(\psi^2-|\D\psi|)_{|M\setminus K}=\psi^2_{|M\setminus K}\equiv\frac{1}{4}\sigma^2n^2\tan^2\Big(\frac{1}{2}\sigma n\delta\Big).
\end{equation}
Consider the Callias operator $\calD_\psi=D+\psi\theta$ on $M$, where $D$ and $\theta$ are given in \eqref{E:twist Dirac} and \eqref{E:GL pair rel Dirac}, respectively. Let $\rho_M$ denote the trivial extension of $\rho$ to $M$. We are concerned about the spectral flow $\spf(\calD_\psi,\rho_M)$.

If we cut $M$ along $\p L$, then we get two manifolds $L$ and $L':=\overline{M\setminus L}$ with compact boundary. By Corollary~\ref{C:spf Callias} and Lemma~\ref{L:spf vanishing-1},
\[
\spf(\calD_\psi,\rho_M)=\spf(\calD_{\psi,1}^L,\rho)+\spf(\calD_{\psi,1}^{L'},\rho_{L'})=\spf(\calD_{\psi,1}^L,\rho),
\]
where $\rho_{L'}$ is the restriction of $\rho_M$ to $L'$. The right hand side is non-zero by Theorem~\ref{T:GL pair spf} and Lemma~\ref{L:GL pair}.4. Hence, $\spf(\calD_\psi,\rho_M)\ne0$.

On the other hand, let
\[
\calD_\psi(r)=(1-r)\calD_\psi+r\rho_M^{-1}\calD_\psi\rho_M=D(r)+\psi\theta,\quad 0\le r\le1.
\]
So
\[
\begin{aligned}
\calD_\psi^2(r) &=D^2(r)+\psi^2+\rmc(\D\psi)\theta \\
&=(\nabla^*\nabla)(r)+\calR(r)+\psi^2+\rmc(\D\psi)\theta,
\end{aligned}
\]
where $\calR(r)=\frac{1}{4}\scal_g+\calR^{E\oplus F}(r)$. Examining the bundle endomorphism
\[
\frac{n}{n-1}\calR(r)+\psi^2+\rmc(\D\psi)\theta\ge\frac{n}{n-1}\calR(r)+\psi^2-|\D\psi|,
\]
we have for any $r\in[0,1]$,
\[
\frac{n}{n-1}\calR(r)+\psi^2-|\D\psi|=\frac{n}{n-1}\Big(\frac{1}{4}\scal_g+\calR^{E\oplus F}(r)\Big)\ge 0\mbox{ on }\supp(\D f)
\]
from Lemma~\ref{L:GL pair}.3 and \eqref{E:scal bound supp},
\[
\frac{n}{n-1}\calR(r)+\psi^2-|\D\psi|=\frac{n}{n-1}\cdot\frac{1}{4}\scal_g+\psi^2-|\D\psi|\ge0\mbox{ on }K\setminus\supp(\D f)
\]
from \eqref{E:scal bound K} and \eqref{E:psi^2-dpsi bound}, and
\[
\frac{n}{n-1}\calR(r)+\psi^2-|\D\psi|=\frac{n}{n-1}\cdot\frac{1}{4}\scal_g+\psi^2\ge0\mbox{ on }M\setminus K
\]
from \eqref{E:scal ge tan^2} and \eqref{E:psi on M-K}.
Now by Remark~\ref{R:sp est Callias}, for any $u\in\ker(\calD_\psi(r))$, we have
\[
0\ge\int_M\Big\langle u,\Big(\frac{n}{n-1}\calR(r)+\psi^2-|\D\psi|\Big)u\Big\rangle \D V_M.
\]
Again as mentioned in the proof of Theorem~\ref{T:neck length}, one can deduce that $u$ vanishes almost everywhere on $M$, which means that $\spf(\calD_\psi,\rho_M)=0$---a contradiction! This proves the theorem for $n$ odd.

When $n$ is even, one considers instead the index of a single Callias operator from the relative Dirac bundle of the form \cite[Example~2.5]{CeccZeid24GT} associated to the same Gromov--Lawson pair. The computation and argument are essentially the same as above (and indeed a little simpler). To sum up, the theorem is proved.
\end{proof}

\begin{remark}\label{R:quanti Llarull}
Theorem~\ref{IT:quanti Llarull} can be viewed as a codimension zero analogue of \cite[Theorem~3.1]{Zeidler20}. Notice that we rule out the equality case in our conclusion.
\end{remark}

Theorem~\ref{IT:quanti Llarull} now indicates a new proof of Llarull's theorem on non-compact manifolds.

\begin{proof}[Proof of Theorem~\ref{IT:Llarull non-cpt}]
By the hypothesis
\[\scal_g\ge n(n-1)\mbox{ on }\supp(\D f),
\]
one can find a compact subset $K\Subset M$ containing $\supp(\D f)$ with smooth boundary such that $0<\delta:=\dist_g(\supp(\D f),\p K)<\frac{2\pi}{n}$ and
\[
\scal_g\ge\frac{1}{4}n(n-1)\mbox{ on }K.
\]
Then by Theorem~\ref{IT:quanti Llarull},
\[
\inf(\scal_g)<-\frac{1}{4}n(n-1)\tan^2\Big(\frac{1}{4}n\delta\Big)<0.
\]
\end{proof}

\begin{remark}\label{R:Llarull non-cpt}
Theorem~\ref{IT:Llarull non-cpt} can also be proved in a more accessible way without using Callias operators with \emph{Lipschitz} potential. Basically, having a compact subset $K$ as above, by modifying the distance function, it is possible to find a \emph{smooth} function $\psi:M\to[0,\varepsilon]$ for some $\varepsilon>0$ depending on $\delta$ such that
\begin{enumerate}
\item $\psi\equiv0$ on $\supp(\D f)$, \label{item:Llarull non-cpt 1}
\item $\psi\equiv\varepsilon$ outside $K$, and \label{item:Llarull non-cpt 2}
\item $\psi^2-|\D\psi|\ge-\frac{1}{8}n(n-1)$ on $K\setminus\supp(\D f)$. \label{item:Llarull non-cpt 3}
\end{enumerate}
Now use such $\psi$ as the potential of Callias operators instead. Again by Lemma~\ref{L:GL pair} and the hypothesis on the lower bound of the scalar curvature, one can show that the bundle endomorphism $\calD^2_\psi(r)-(\nabla^*\nabla)(r)$ is non-negative on the whole of $M$ and has a positive lower bound $\varepsilon\cdot\id$ on $M\setminus K$ for any $r\in[0,1]$. Thus the usual contradiction argument works. The idea is similar to that in \cite{Zhang20}.
\end{remark}

\section*{Funding}
This work is supported in part by the NSFC (grant no. 12101042) and Beijing Institute of Technology Research Fund Program for Young Scholars.

\section*{Acknowledgments}
The author is grateful to the referee for helpful comments and suggestions that improve the paper.

\section*{Author contributions statement}
P.S. conducted the research and wrote the manuscript.

\section*{Competing interests}
No competing interest is declared.

\bibliographystyle{amsplain}


\end{document}